\documentclass[12pt]{amsart}
\nonstopmode \numberwithin{equation}{section}
\usepackage{amsfonts}
\usepackage{amssymb}
\usepackage{ifthen}
\usepackage{amsxtra}
\usepackage{csquotes}

\theoremstyle{plain}

\newtheorem*{theorem*}{Theorem}

\newtheorem{thm}{Theorem}

\newtheorem{cor}{Corollary}

\theoremstyle{definition}
\newtheorem{definition}[thm]{Definition}
\newtheorem{example}{Example}

\newtheorem{prop}{Proposition}[section]
\newtheorem{remark}[thm]{Remark}

\newtheorem{theorem}{Theorem}
\newtheorem{lemma}{Lemma}
\newcounter{alphabet}

\newenvironment{Thm}[1][]{\refstepcounter{alphabet}%
	\bigskip%
	\noindent%
	{\bf Theorem \Alph{alphabet}}%
	\ifthenelse{\equal{#1}{}}{}{ (#1)}%
	{\bf .} \itshape}{\vskip 8pt}

\usepackage{amsmath}

\newcommand{\blem}{\begin{lemma}}
\newcommand{\elem}{\end{lemma}}
\newcommand{\bexer}{\begin{exe}}
\newcommand{\eexer}{\end{exe}}
\newcommand{\beq}{\begin{eqnarray}}
\newcommand{\eeq}{\end{eqnarray}}
\newcommand{\bthm}{\begin{theorem}}
\newcommand{\ethm}{\end{theorem}}
\newcommand{\bex}{\begin{example}}
\newcommand{\eex}{\end{example}}

\newcommand{\bdefi}{\begin{definition}}

\newcommand{\edefi}{\end{definition}}

\newcommand{\bprop}{\begin{prop}}
\newcommand{\eprop}{\end{prop}}

\newcommand{\bpf}{\begin{proof}}
\newcommand{\epf}{\end{proof}}

\def\be{\begin{equation}}
\def\ee{\end{equation}}
\newcommand{\ds}{\displaystyle}

\newcommand{\U}{\mathbb{U}}

\newcommand{\fracc}[2]{\mbox{\small$\displaystyle\frac{#1}{#2}$}}

\numberwithin{equation}{section}

\input amssym.def
\input amssym

\begin{document}
\title[Schwarz lemma]{ Schwarz lemmas for mappings with bounded Laplacian}

\author{Miodrag Mateljevi\'c}
\address{M. Mateljevi\'c, Faculty of mathematics, University of Belgrade, Studentski Trg 16, Belgrade, Republic of Serbia}
\email{miodrag@matf.bg.ac.rs}

\author{Adel Khalfallah}
\address{A. Khalfallah, King Fahd University of Petrolem and Minerals, Department of Mathematics and Statistics, Dhahran,
	Kingdom of Saudi Arabia.} \email{khelifa@kfupm.edu.sa}


\begin{abstract}
	
	We establish some Schwarz type Lemmas for   mappings  defined on the unit disk with bounded Laplacian.  Then we apply these
	results to obtain boundary versions of the Schwarz lemma.
\end{abstract}

\maketitle \pagestyle{myheadings} \markboth{ M. Mateljevi\'c and A.
	Khalfallah}{ Schwarz  Lemmas  for mappings with bounded Laplacian}
\section{Introduction and preliminaries}
Motivated by the role of the Schwarz lemma in complex analysis and
numerous fundamental results, see for instance  \cite{mmjmaa2018,osserman} and references  therein,
in 2016, the first author \cite{RgSchw1}(a) has posted on ResearchGate the   project \enquote{Schwarz lemma, the Carath\'{e}odory and Kobayashi Metrics and Applications in Complex Analysis}\footnote{Motivated by  S. G. Krantz  paper \cite{krantz}.}.  Various discussions regarding the subject can also  be found  in the Q\&A section on ResearchGate under the question \enquote{What are the most recent versions of the Schwarz lemma ?}\cite{RgSchw1}(b)\footnote{The subject has been presented at Belgrade analysis seminar \cite{Bg_sem}.}. In this project  and  in  \cite{mmjmaa2018}, cf. also \cite{kavu}, we developed  the method  related to  holomorphic mappings with strip codomain (we refer to this method as  the approach  via the Schwarz-Pick lemma for holomorphic maps from the unit disc into a strip). It is worth mentioning that the Schwarz  Lemma has been generalized in various directions, see \cite{AzerOr0,kavu,Khal,MmSm1,zhu}.

Recently  X. Wang, J.-F. Zhu  \cite{XJ} and   Chen and Kalaj \cite{ChKa}  have studied boundary Schwarz lemma for solutions to Poisson's equation.
They improved  Heinz's theorem \cite {Heinz}  and   Theorem A below. We found that Theorem A  is a forgotten  result   of    H.  Hethcote \cite{heth}, published in 1977.

Note that  previously  B. Burgeth \cite{bu}   improved the above result of Heinz  and  Theorem \Alph{alphabet} for real valued  functions (it is easy to extend his result for complex valued  functions, see below) by removing the assumption $f(0)=0$ but it is overlooked in the literature.
Recently, M. Mateljevi\'c and M. Svetlik \cite{MmSm1} proved a Schwarz lemma for real harmonic functions with values in $(-1,1)$  using a completely different  approach than B. Burgeth \cite{bu}.
In this paper, we further develop the  method initiated in  \cite{MmSm1}. More precisely, we show if $f: \U \to (-1,1)$,   $f \in C^2(\mathbb{U})$ and continuous on  $\overline{\mathbb{U}}$, and   $|\Delta f|\leq  c$ on $\mathbb{U}$ for some $c>0$ then, the mapping  $\ds u=f \pm   \frac{c}{4} (1- |z|^2)$  is subharmonic or superharmonic and we  estimate the harmonic function  $P[u^*]$, see Theorem \ref{th3}. Next, we extend the previous result to complex valued functions, see Corollary \ref{cor4}. Finally, we establish  Schwarz Lemmas at the boundary for solutions to  $|\Delta f| \leq c $.  Our results are generalizations of Theorem 1.1 \cite{XJ} and Theorem 2 \cite{ChKa}.

\subsection{Notations and background}

In this paper  $\mathbb{U}$ denotes the open  unit disk $\{z \in \mathbb{C}\,:\, |z|<1 \}$, and $\mathbb{T}$ denotes the unit circle.

A real-valued function $u$, defined in an open subset $D$ of the complex plane $\mathbb{C}$, is harmonic if it satisfies Laplace's equation
$$\Delta u (z)= \frac{\partial^2 u}{\partial x^2}(z)+  \frac{\partial^2 u}{\partial y^2}(z)=0 \quad \quad (z=x+iy) .$$

A real-valued function $u \in \mathcal{C}^2(D)$ is called subharmonic if $\Delta u(z) \geq 0$ for all $z\in D$.

\noindent A complex-valued function $w=u+iv$ is harmonic if both $u$ and $v$ are real harmonic.\\

Let $P$ be the Poisson kernel, i.e., the function
$$ P(z,e^{i\theta})=\frac{1-|z|^2}{|z-e^{i\theta}|^2},$$
and let $G$ be the Green function on the unit disk, i.e., the function
$$G(z,w)=\frac{1}{2\pi} \log \left|\frac{1-z\overline{w}}{z-w} \right| \quad z,w\in \mathbb{U}, z \not = w. $$
 Let $\phi\in L^1(\mathbb{T})$  be an integrable function on the unit circle, then the function $P[\phi]$ given by
$$P[\phi](z)=\frac{1}{2\pi}\int_{0}^{2\pi} P(z,e^{i\theta}) \phi(e^{i\theta}) d\theta$$
 is harmonic in $\U$ and has a radial limit that agrees with $\phi$ almost everywhere on the boundary $\mathbb{T}$ of the unit disc.\\

For $g\ \in \mathcal{C}(\overline{\U})$, let
$$G[g](z)=\int_{\mathbb{U}} G(z,w) g(w) dm(w)$$
$|z|<1$ and $dm(w)$ denotes the Lebesgue measure in $\mathbb{U}$.\\

If we consider the function
 $$u(z):=P[\phi](z)-G[g](z), $$
 then $u$ satisfies the Poisson's equation
 $$
 \begin{cases}
 	\Delta u=g \mbox { on the disk }\mathbb{U},\\
 	\ds \lim_{r\to 1^-}u(re^{i\theta})=\phi(e^{i\theta}) \mbox{ a.e. on the circle.}
 	 \end{cases}
  $$

One can easily see that  the previous equation has  {\it non-unique} solution. Indeed, the Poisson kernel $P(z)=\ds \frac{1-|z|^2}{|1-z|^2}$ is a harmonic function on the unit disk and $\ds \lim_{r\to 1^-}P(re^{i\theta})=0$ a.e., but  $P \not = 0$. \\

It is well known that  if $\phi$ is continuous on the unit circle, then the harmonic function $P[\phi]$  extends continuously on $\mathbb{T}$ and equals to $\phi$ on $\mathbb{T}$, see H\"ormander \cite{hor}.\\

The following is a consequence of the maximum principle  for subharmonic functions.
\begin{theorem*}[Harmonic Majoration]
	
	Let $u$ be a subharmonic function  in $\mathcal{C}^2(\mathbb{U})\cap \mathcal{C}(\overline{\U})$. Then
	$$u \leq P[u\vert_{\mathbb{T}}] \mbox { on } \mathbb{U}.$$
\end{theorem*}

{\it Notations:} For $b\in (-1,1)$, and $r\in[0,1]$ let us denote

$$A_b(r):=\fracc{1-r^2}{1+r^2}b+\fracc{4}{\pi} \arctan r,$$

$$B_b(r):=\fracc{1-r^2}{1+r^2}b-\fracc{4}{\pi} \arctan r.$$
 $$ M_b(r):=\fracc{4}{\pi} \arctan \fracc{a+r}{1+ar},$$
$$m_b(r):=M_b(-r)=\fracc{4}{\pi} \arctan \fracc{a-r}{1-ar},$$

where $$a=\tan \frac{b\pi}{4}.$$

Finally, if $f\in \mathbb{C}^2(\mathbb{U}) \cap \mathcal{C}(\overline{\mathbb{U}})$, then $f^*:=f\vert_{\mathbb{T}}$.

\section{The Schwarz Lemma for harmonic functions from $\mathbb{U}$ into $\mathbb{U}$}

The classical Schwarz Lemma (Heinz's theorem \cite {Heinz}) for  a harmonic mapping $f$ from the unit disc to the unit such that $f(0)=0$ is given by
$$|f(z)| \leq \frac{4}{\pi} \arctan |z|.$$
Moreover, this inequality is sharp for each point $z\in \mathbb{U}$.\\

Later, in 1977, H.  Hethcote \cite{heth}
improved the above result of Heinz by removing the assumption $f(0)=0$ and showed the following:

\begin{Thm} \cite{heth,Pavlovic}
If $f$ is a harmonic mapping from the unit disc to the unit disc, then
$$ \left|f(z) -\frac{1 - |z|^2}{1 + |z|^2} f(0)\right| \leq \frac{4}{\pi} \arctan |z|.$$
\end{Thm}

Then  B. Burgeth \cite{bu}   improved the above result of Heinz  and  Theorem \Alph{alphabet} for real valued  functions (it is easy to extend his result for complex valued  functions, see below) by removing the assumption $f(0)=0$.\\

First, we remark that the estimate of Theorem \Alph{alphabet}  cannot be sharp for all values $z$ in  the unit disc. Indeed,    let $f$ be a {\it real} harmonic functions with codomain $(-1,1)$ and  let $b=f(0)$. A
simple study of the function $A_b$, defined by  $A_b=\fracc{1-r^2}{1+r^2}b+\fracc{4}{\pi} \arctan r$, shows that if $b>\frac{2}{\pi}$, then there exists $r_0$ unique  in $(0,1)$  such that $A_b(r_0)=1$. Moreover  $\displaystyle\max_{r\in[0,1]}A_b(r)>1$. Similarity, for $b<-\frac{2}{\pi}$, there exists $r_1$ unique  in $(0,1)$  such that $B_b(r_1)=-1$   and $\displaystyle\min_{r\in[0,1]}B_b(r)<-1$, where $B_b(r):=\fracc{1-r^2}{1+r^2}b-\fracc{4}{\pi} \arctan r$.\\

Recently, M. Mateljevi\'c and M. Svetlik \cite{MmSm1} proved a Schwarz lemma for real harmonic functions with values in $(-1,1)$  using completely different  approach than B. Burgeth  \cite{bu}.

\begin{Thm}\cite{MmSm1}
Let $u:\mathbb{U}\to (-1,1)$ be a harmonic function such that $u(0)=b$. Then
$$m_b(|z|) \leq u(z) \leq M_b(|z|) \, \mbox{ for all } z\in \mathbb{U}. $$
 This inequality is sharp for each point $z\in \mathbb{U}$.
\end{Thm}

Let us recall that,  $ M_b(r)=\fracc{4}{\pi} \arctan \fracc{a+r}{1+ar}, m_b(r)=\fracc{4}{\pi} \arctan \fracc{a-r}{1-ar}$
and,

$a=\displaystyle \tan \frac{b\pi}{4}.$\\

Clearly this result improves   Theorem A  for real harmonic functions, since
\begin{prop}\label{prop1}
Let $b$ be a real number in $(-1,1)$. Then
\begin{equation}
M_b \leq A_b \quad \mbox{ and  }\quad  B_b \leq m_b \quad \mbox{ on } [0,1].
\end{equation}
\end{prop}

That is for any $r\in [0,1)$ and any $b\in (-1,1)$, we have

$$ \fracc{4}{\pi} \arctan \fracc{a+r}{1+ar} \leq \fracc{1-r^2}{1+r^2}b+\fracc{4}{\pi} \arctan r, $$
and,

$$ \fracc{1-r^2}{1+r^2}b- \fracc{4}{\pi} \arctan r \leq \fracc{4}{\pi} \arctan \fracc{a-r}{1-ar},$$
where $a=\displaystyle\tan \frac{b\pi}{4}$.

\begin{proof}

Let us show the first inequality $ \fracc{4}{\pi} \arctan \fracc{a+r}{1+ar} \leq \fracc{1-r^2}{1+r^2}b+\fracc{4}{\pi} \arctan r$, that is:  $$  \arctan \fracc{a+r}{1+ar}-\arctan r \leq \fracc{1-r^2}{1+r^2}(b\fracc{\pi}{4}) $$ which is equivalent to
$$\arctan\left( \frac{a(1-r^2)}{r^2+2ar+1} \right) \leq \frac{1-r^2}{1+r^2} \arctan (a). $$

Let fix $r\in(0,1)$ and consider the mapping $$\phi(a)= \arctan\left( \frac{a(1-r^2)}{r^2+2ar+1} \right) - \frac{1-r^2}{1+r^2} \arctan (a).$$
The derivative of $\phi$ is equal to $$\phi'(a)=\dfrac{4r\left(r^2-1\right)a}{\left(r^2+1\right)\left(a^2+1\right)\left(\left(r^2+1\right)a^2+4ra+r^2+1\right)}.$$

Thus   $\phi$ has a maximum at $a=0$. As $\phi(0)=0$, we get  $\phi(a)\leq 0$. 	
	\end{proof}

Combining Theorem \Alph{alphabet} and Proposition \ref{prop1}, we get
\begin{cor}\label{cor1}
If $u:\U \to (-1,1)$ is harmonic, then
\begin{equation}\label{eq1.2}
u(z)-\frac{1 - |z|^2}{1 + |z|^2}\, b \leq M_b(|z|)-\frac{1 - |z|^2}{1 + |z|^2} \, b\leq \frac{4}{\pi} \arctan |z|.
\end{equation}
and
\begin{equation}\label{eq1.3}
u(z)-\frac{1 - |z|^2}{1 + |z|^2} \, b  \geq m_b(|z|)-\frac{1 - |z|^2}{1 + |z|^2} \, b \geq -\frac{4}{\pi} \arctan |z|.
\end{equation}

where $b=u(0)$.
\end{cor}

\begin{remark} One may obtain Proposition \ref{prop1}   using    Theorem  \Alph{alphabet}.

Indeed,
 for $b\in (-1,1)$, let ${\rm Har}(b)$    denote the family of all real  valued   harmonic maps
$u$ from $\mathbb{U}$ into $(-1,1)$ with $u(0)=b$.  For   $z\in \mathbb{U}$,    set   $M_b^*(z)=\sup \{u(z): u\in {\rm Har}(b) \}$.

If  $v \in {\rm Har}(b)$, then   we have   $v(r)\leq  A_b(r)$, we call  $A_b$ a majorant for  ${\rm Har}(b)$.
It is clear that  $ M_b^* \leq  A_b$.

From  the proof  of  Theorem 5 \cite{MmSm1}, it follows that   $M_b^*(r)=M_b(r)$ on $[0,1)$.  Hence  $ M_b \leq  A_b$  on $[0,1)$.\\
\end{remark}

Our first aim is to extend Theorem \Alph{alphabet} for  {\it complex} harmonic  functions from the unit disc to the unit disc, and to recover Theorem  A.

\begin{theorem} Let $f: \U \to \U$ be a harmonic function from the unit disc to the unit disc. Then
	$$ |f(z)| \leq M_b(|z|), \mbox { where } b=|f(0)|.$$
\end{theorem}

\begin{proof}
(1) Fix $z_0$ in the unit disc and choose unimodular $\lambda= e^{i \alpha}$ such that $\lambda f(z_0)>0$.

Define $u(z)=\Re (\lambda f(z))$;  $u(z_0)=|f(z_0)|$.

Hence, using Theorem \Alph{alphabet}, we get
$$|f(z_0)| \leq u(z_0) \leq M_{u(0)}(|z_0|) \leq M_{|f(0)|} (|z_0|).$$
As the mapping  $b\mapsto A_b(|z_0|)$ is increasing.
\end{proof}

\begin{remark}\hskip0.1\textwidth
	
\noindent One can easily recover the estimate  	$ \displaystyle \left|f(z) -\frac{1 - |z|^2}{1 + |z|^2} f(0)\right| \leq \frac{4}{\pi} \arctan |z|$ obtained in Theorem A.

Indeed,  let $u(z)=\Re(\lambda f(z))$. By equations \ref{eq1.2} and \ref{eq1.3}, we have
$$\ds \left|u(z) -\frac{1 - |z|^2}{1 + |z|^2} u(0)\right| \leq \frac{4}{\pi} \arctan |z|.$$
But
$$\ds u(z) -\frac{1 - |z|^2}{1 + |z|^2} u(0)= \Re \left(\lambda \displaystyle \left(f(z) -\frac{1 - |z|^2}{1 + |z|^2} f(0)\right)\right).$$ Now, one can choose $\lambda$, such that $\ds u(z) -\frac{1 - |z|^2}{1 + |z|^2} u(0)= \left|f(z) -\frac{1 - |z|^2}{1 + |z|^2} f(0)\right|$.
\end{remark}
\section{Schwarz Lemma for mappings with bounded Laplacian}
In   \cite{ChKa},   for a given continuous function  $g: G\rightarrow \mathbb{C}$, Chen and Kalaj  establish some Schwarz type Lemmas for mappings $f$  in $G$ satisfying the Poisson's equation $\Delta f=g$, where $G$ is a subset of the complex plane $\mathbb{C}$. Then they apply these results to obtain a Landau type theorem, which is a partial answer to the open problem in \cite{CP}.
Set  $A(z)= \displaystyle\frac{1-|z|^2}{1+|z|^2} $.\\

  We claim   that  Theorem 1  in  Chen and Kalaj paper   \cite{ChKa}  holds under additional conditions. Here we suggest  a modification of this result.

\begin{Thm}[]
	Let   $g\in C(\overline{\mathbb{U}})$ and $\phi \in L^{\infty}(\mathbb{T})$.
	Then the function
	$$f:= P[\phi]-G[g]$$
	 is $C^2(\mathbb{U})$ and satisfies the Poisson's equation
	 $$
	\begin{cases}
	\Delta f=g ,\\
    f\vert_{S^1}=\phi.
	\end{cases}
	$$
	
	In addition, we   have
	\begin{equation} \label{3.1}
    |f(z)-  P[\phi](0) A(z)|\leq \frac{4}{\pi} ||P[\phi]||_\infty \arctan |z| +  \frac{1}{4} ||g||_\infty (1-|z|^2),
    \end{equation}
    where $ ||P[\phi]||_\infty=\sup_{z\in \mathbb{U}}|P[\phi](z)|$ and $||g||_\infty =\sup_{z\in\mathbb{D}}|g(z)|.$ By $ f\vert_{S^1}=\phi$, we mean that $\ds \lim_{r \to 1^-} f(re^{i\theta})=\phi(e^{i\theta})$  a.e. \\
\end{Thm}
\begin{remark}
	Let $g\in C(\overline{\mathbb{U}})$ and $\phi \in C(\mathbb{T})$. Then the solutions of the Poisson's equation
	
	 $$
	\begin{cases}
	\Delta f=g ,\\
	f\vert_{S^1}=\phi,
	\end{cases}
	$$
	are given by 	$f:= P[\phi]-G[g]$, see \cite{hor} p. 118-120, and we have the estimate (\ref{3.1}). Here the boundary condition means that $\ds\lim_{z\in \U, z\to \xi} f(\xi)=\phi(\xi)$, for all $\xi \in \mathbb{T}$.  Thus  the version  stated in  Chen and Kalaj   (Theorem 1 in  \cite{ChKa})  holds in this setting.\\
\end{remark}


 Now   we    show that Theorem  \ref{th3} (the next result below) implies Theorem \Alph{alphabet}.

Let $f$ be a {\it real} valued  $ C^2(\mathbb{U})$ continuous on $\overline{\U}$,  such that $\Delta f=g$ and $f^*=f\vert_{\mathbb{T}}$. Let $K := ||P[f^*]||_\infty$. By Theorem \ref{th3}, we have
$$  m_{b/K}(|z|) K -   ||g||_\infty \frac{(1- |z|^2)}{4}   \leq f(z)\leq M_{b/K}(|z|) K +   ||g||_\infty \frac{(1- |z|^2)}{4},$$
where $b=P[f^*](0)$. Using Proposition \ref{prop1}, we get
  $$M_{b/K} (|z|)K  \leq \frac{1-|z|^2}{1+|z|^2}b+\frac{4K}{\pi}   \arctan |z|,$$
 $$m_{b/K}(|z|)K  \geq \frac{1-|z|^2}{1+|z|^2}b-\frac{4K}{\pi} \arctan |z|.$$
 Hence

   $$\left|f(z)-    \frac{1-|z|^2}{1+|z|^2} P[f^*](0) \right|\leq \frac{4}{\pi} ||P[f^*]||_\infty \arctan |z| +  \frac{1}{4} ||g||_\infty (1-|z|^2)$$

If $f$ is complex valued function, we consider $u=\Re (\lambda f)$, where $\lambda$ is a complex number of modulus $1$.

\begin{thm}\label{th3}

    Let $f$ be $ C^2(\mathbb{U})$ real-valued function, continuous on $\overline{\U}$.  Let $b=P[f^*](0)$, $c\in \mathbb{R}$ and $K$ is a positive number such that $K \geq ||P[f^*]||_\infty$, where $f^*=f\vert_{S^1}$.
\begin{itemize}
    \item[(i)] If $f$ satisfies    $\Delta f\geq - c$, then
     $$f(z)\leq K M_{b/K}(|z|) +   \frac{c}{4} (1- |z|^2).$$

   \item[(ii)] If $f$ satisfies $\Delta f\leq  c$, then
   $$ f(z) \geq K m_{b/K}(|z|) -   \frac{c}{4} (1- |z|^2).$$
\end{itemize}
\end{thm}

\begin{proof}

(i) Define   $\displaystyle f^0(z)= f(z)  + \frac{c}{4} (|z|^2-1)$, and set $P[f^*](0)=b$.  Then   $f^0$ is subharmonic   and    $f^0 \leq P[f^*]$. As $\displaystyle \frac{1}{K} P[f^*]$ is a real harmonic function with codomain $(-1,1)$, by Theorem 5 in \cite{MmSm1}, we obtain $P[f^*](z) \leq K M_{b/K}(|z|)$. Thus
$$f(z)\leq K M_{b/K}(|z|) +   \frac{c}{4} (1- |z|^2), \mbox{ for all }  z\in \mathbb{U}.$$

(ii) If $f$ satisfies   $\Delta f\leq  c$, then     define   $\displaystyle f_0(z)= f(z)  - \frac{c}{4} (|z|^2-1)$, and set $P[f^*](0)=b$.
In a similar way, using Theorem 5 \cite{MmSm1}, we show that for all $z\in\mathbb{U}$, we get
 $$ f(z) \geq K m_{b/K}(|z|) -   \frac{c}{4} (1- |z|^2).$$
\end{proof}


\begin{cor}\label{cor4}
Suppose that $f: \U \to \U$,   $f \in C^2(\mathbb{U})$ and continuous on  $\overline{\mathbb{U}}$, and   $|\Delta f|\leq  c$ on $\mathbb{U}$ for some $c>0$. Then
$$|f(z)|\leq M_b(|z|) +   \frac{c}{4} (1- |z|^2),$$
 where $b=|P[f^*](0)|$.
\end{cor}

\begin{proof}
Fix $z_0$ in the unit disc and choose $\lambda$ (depends on $z_0$) such that $\lambda f(z_0)=|f(z_0)|$. Define
$u(z)=\Re (\lambda f(z))$ (we say that  $u$ is real valued  harmonic   associated to complex  valued  harmonic   $f$  at $z_0$).
We have $\Delta u= \Re(\lambda \Delta f)$. As $u$ is a real function with codomain $(-1,1)$ satisfying  $|\Delta u|\leq  c$, by Theorem \ref{th3}, we get
$$|u(z)|\leq M_{b_1}(|z|) +   \frac{c}{4} (1- |z|^2), \mbox{ where } b_1=P[u^*](0).$$
We have $b_1=P[u^*](0)=\Re(\lambda P[ f^*](0))\leq |P[f^*](0)|$.
Hence
$$|f(z_0)|\leq M_b(|z_0|) +   \frac{c}{4} (1- |z_0|^2),$$
where $b=|P[f^*](0)|$, as the mapping  $b\mapsto A_b(|z_0|)$ is increasing.
\end{proof}

Similarly, under the conditions of the previous theorem, we obtain
\begin{equation}\label{eq2.2}
\left|f(z)-  P[f^*](0) A(z)\right|\leq \frac{4}{\pi}  \arctan |z| +  \frac{c}{4} (1-|z|^2)
\end{equation}

\section{Boundary Schwarz Lemmas}

We establish  Schwarz Lemmas at the boundary for solutions to  $|\Delta f| \leq c $.  Our results are generalizations of Theorem 1.1 \cite{XJ} and Theorem 2 \cite{ChKa}.

\begin{thm}
	Suppose  $f\in C^2(\U)$, continuous on $\overline{\U}$  with codomain $(-1,1)$,    such that $\Delta f \geq -c$. If $f$ is differentiable at $z=1$ with  $f(1) =1$, then the following inequality holds:
	
$$f_x(1)\geq   \frac{2}{\pi} \tan\frac{\pi}{4}(1-b) -\frac{c}{2}.$$

where $$  b=P[f^*](0).$$\\
\end{thm}

Before giving the proof, one can easily  show that
$$M_b'(r)=\frac{4}{\pi}\left[\dfrac{1-a^2}{\left(a^2+1\right)r^2+4ar+a^2+1}\right]$$
Hence
$$M_b'(1)= \frac{2}{\pi} \left[\frac{1-a}{1+a}\right]=\frac{2}{\pi} \tan\frac{\pi}{4}(1-b).$$

As $a=\ds\tan \frac{b\pi}{4}.$

\begin{proof}
Since $f$ is differentiable at $z=1$, we know that
$$f(z)=1+f_z(1)(z-1)+f_{\bar{z}}(1)(\bar{z}-1)+ o(|z -1|).$$
 That is
 $$f_x(1)= \lim_{r \to 1^-} \frac{f(r)-1}{r-1}.$$

On the other hand, Theorem \ref{th3}(i) leads to

$$ 1-M_b(r)  -\frac{c}{4} (1- r^2)  \leq 1-f(r).$$

Dividing by $(1-r)$ and letting  $r\to 1^-$, we get

\begin{equation}\label{eq1}
f_x(1) \geq M_b'(1)-\frac{c}{2}.
\end{equation}
Thus
$$  f_x(1) \geq  \frac{2}{\pi} \tan\frac{\pi}{4}(1-b)  -\frac{c}{2}.$$
\end{proof}

\begin{cor}
	Suppose  $f\in C^2(\U)$, continuous on $\overline{\U}$ with codomain $(-1,1)$   and  differentiable at $z=1$ with  $f(1) =1$.
	\begin{itemize}
		\item[(i)]  If $\Delta f \geq -c$ then,
		$$f_x(1) \geq \displaystyle-b+\frac{2}{\pi}-\frac{c}{2}.$$
		\item[(ii)]  If  $|\Delta f| \leq c$ and $f(0)=0$, then $\displaystyle |b| \leq \frac{c}{4}$ and
		
		$$f_x(1) \geq -\frac{3}{4}c +\frac{2}{\pi}.$$
	\end{itemize}
	
	\end{cor}
	
	\begin{proof}
	(ii)	Using $M_b \leq A_b$  and $M_b(1)=A_b(1)=1$, we get
		
		$$ M_b'(1) \geq A_b'(1) = -b +\frac{2}{\pi},$$
	(ii) The estimate $\displaystyle |b| \leq \frac{c}{4}$ follows directly from Theorem \ref{th3}.

	\end{proof}
	
	\begin{remark}
		One can also prove directly that $M_b'(1) \geq A_b'(1)$, that is
		\begin{equation}\label{conv}
		 \frac{2}{\pi} \tan\frac{\pi}{4}(1-b) \geq -b+\frac{2}{\pi} \mbox { for } b\in[0,1).
		\end{equation}
		Using the convexity of the tangent function, we get
		$$\tan x \geq 2(x-\frac{\pi}{4})+1 \mbox { for }  x\in[0,\pi/2).$$
		For $b\in [0,1)$, let us substitute  $x$ by $\frac{\pi}{4}(1-b)$, we obtain
		$$\frac{2}{\pi} \tan\frac{\pi}{4}(1-b) \geq -b+\frac{2}{\pi}.$$
	\end{remark}

	The following theorem is a generalization on Theorem 2 in \cite{ChKa} where the authors proved a Schwarz Lemma on the boundary for functions satisfying $\Delta f=g$ and under the assumption $f(0)=0$.
\begin{thm}
Suppose  $f\in \mathcal{C}^2(\U)\cap \mathcal{C}(\overline{\U})$ is a function of $\U$ into $\U$ satisfying satisfying $|\Delta f| \leq c$, where   $\ds 0\leq c<\frac{4}{\pi} \tan\frac{\pi}{4}(1-b) $. If  for some $\xi \in \mathbb{T}$, $\ds \lim_{r \to 1^-} |f(r\xi)|=1$, then

$$\liminf_{r \to 1^-} \frac{|f(\xi)-f(r\xi)|}{1-r} \geq  \frac{2}{\pi} \tan \frac{\pi}{4}(1-b) -\frac{c}{2}.$$

where $  b=|P[f^*](0)|.$\\

If, in addition, we assume that $f(0)=0$, then,
 $$\displaystyle |b| \leq \frac{c}{4}$$ and
$$\liminf_{r \to 1^-} \frac{|f(\xi)-f(r\xi)|}{1-r}  \geq -\frac{3}{4}c +\frac{2}{\pi}.$$
\end{thm}

\begin{proof}
Using  Corollary  \ref{cor4}, we have
$$|f(\xi)-f(r\xi)|\geq 1-|f(r\xi)|\geq 1-M_b(r)-\frac{c}{4} (1-r^2).$$	
Thus
$$\liminf_{r \to 1^-} \frac{|f(\xi)-f(r\xi)|}{1-r} \geq \lim_{r \to 1^-} \frac{ 1-M_b(r)-\frac{c}{4} (1-r^2)}{1-r}=M_b'(1)-\frac{c}{2}.$$
The conclusion follows as $M_b'(1)= \ds\frac{2}{\pi} \tan\frac{\pi}{4}(1-b).$\\

If in addition, we assume that $f(0)=0$, using Equation \ref{eq2.2}, we obtain $\displaystyle |b|<\frac{c}{4}$. Hence

$$\liminf_{r \to 1^-} \frac{|f(\xi)-f(r\xi)|}{1-r} \geq   \frac{2}{\pi} \tan\frac{\pi}{4}(1-b)-\frac{c}{2} \geq -b+\frac{2}{\pi}-\frac{c}{2}\geq -\frac{3}{4}c +\frac{2}{\pi}.$$
The second estimates follows from inequality (\ref{conv}).
\end{proof}

\section*{Acknowledgement}

The  first author is partially supported by Ministry of Education, Science and Technological Development of the Republic of Serbia, Grant No. 174 032.

The second author would like to acknowledge the support provided by the Deanship of Scientific Research (DSR) at the King Fahd University of Petroleum and Minerals (KFUPM) for funding this work through project No IN161056.

\end{document}